\documentclass[onefignum,onetabnum]{siamart251216}

\usepackage{mathtools,bm,amsfonts,amssymb}
\usepackage{microtype}
\usepackage{graphicx}
\usepackage[percent]{overpic}
\usepackage{tikz}
\usetikzlibrary{patterns}
\usepackage{booktabs}
\usepackage{array}
\usepackage{enumitem}
\usepackage{xcolor}
\usepackage{comment}

\ifpdf
\hypersetup{
  pdftitle={Oblivious Subspace Injection Is Not Enough for Relative Error},
  pdfauthor={Alex Townsend and Chris Wang}
}
\fi

\allowdisplaybreaks

\crefname{theorem}{Theorem}{Theorems}
\crefname{proposition}{Proposition}{Propositions}
\crefname{corollary}{Corollary}{Corollaries}
\crefname{lemma}{Lemma}{Lemmas}
\crefname{definition}{Definition}{Definitions}
\crefname{remark}{Remark}{Remarks}
\crefname{section}{Section}{Sections}
\crefname{table}{Table}{Tables}
\crefname{figure}{Figure}{Figures}

\newcommand{\R}{\mathbb{R}}
\newcommand{\E}{\mathbb{E}}
\newcommand{\Prob}{\mathbb{P}}
\newcommand{\norm}[1]{\left\lVert #1 \right\rVert}

\newcommand{\range}{\operatorname{range}}
\newcommand{\diag}{\operatorname{diag}}
\newcommand{\rank}{\operatorname{rank}}
\newcommand{\argmin}{\operatorname*{argmin}}
\newcommand{\eps}{\varepsilon}

\headers{Oblivious Subspace Injection Is Not Enough for Relative Error}{Townsend and Wang}
\title{Oblivious Subspace Injection Is Not Enough for Relative Error\thanks{Submitted to the editors \today.
\funding{The work of A.~T. was supported by NSF CAREER (DMS-2045646) and by the Defense Advanced Research Projects Agency (DARPA) through The Right Space (TRS) Disruption Opportunity (DARPA-PA-24-04-07). The work of C.~W. was supported by the NSF GRFP under grant DGE-2139899.
}}}
\author{
Alex Townsend\thanks{Department of Mathematics, Cornell University, Ithaca, NY (\email{townsend@cornell.edu}, \email{cyw33@cornell.edu}).}
\and
Christopher Wang\footnotemark[2]
}

\begin{document}
\maketitle

\begin{abstract}
Oblivious subspace injection (OSI) was introduced by Cama\~no, Epperly, Meyer, and Tropp in 2025 as a much weaker sketching property than oblivious subspace embedding (OSE) that still yields constant-factor guarantees for randomized low-rank approximation and sketch-and-solve least-squares regression.  At the Simons Institute in Berkeley during a workshop in October 2025, it was asked whether OSIs also imply relative error bounds rather than just constant-factor guarantees.  We show that, from a theoretical standpoint, OSI alone does not yield OSE-style relative-error guarantees whose failure probability is controlled solely by the OSI failure parameter, even though OSI sketches often perform extremely well in practice.  We provide counterexamples showing this for sketch-and-solve least squares and for randomized SVD in the Frobenius norm.  The missing ingredient from a sketch satisfying only OSI is upper control on the optimal residual or tail component, and when one ensures the sketch has this additional property, a near-relative-error bound is recovered.  We also show that there is a natural $\ell_p$ analogue of OSI giving constant-factor sketch-and-solve bounds.  
\end{abstract}



\section{Introduction}
Randomized sketching is a central tool in modern numerical linear algebra for reducing the computational cost of large-scale problems; see, e.g., Mahoney~\cite{Mahoney2011}, Woodruff~\cite{Woodruff2014}, and Martinsson--Tropp~\cite{MartinssonTropp2020}.  The basic idea is to compress a high-dimensional dataset into a much smaller representation by multiplying it with a random sketching matrix.  If designed appropriately, this compressed representation preserves the essential geometric structure of the original problem, enabling fast approximate solutions with provable guarantees.  

For a large matrix $A\in\R^{n\times d}$ (with $n>d$) and a vector $b\in\R^n$,  instead of solving the least squares problem $x_\star \in \argmin_{x\in\R^d} \norm{Ax - b}_2$,  one forms a smaller sketched problem. Given a sketch $\Omega\in\R^{n\times k}$, the sketch-and-solve estimator is
\[
\widetilde{x}\in \argmin_{x\in\R^d} \norm{\Omega^\top (Ax-b)}_2.
\]
Likewise, for low-rank approximation, given $A\in\R^{n\times d}$, a target rank $r$, and a sketch $\Omega\in\R^{d\times k}$ with $k\ge r$, one forms the sample matrix $Y=A\Omega$ and the associated rangefinder approximation
\[
\widetilde A = (A\Omega)(A\Omega)^+A,
\]
where ${}^+$ denotes the Moore--Penrose pseudoinverse. The goal in both settings is to replace a large problem by a smaller one while still obtaining a high-quality approximation.

The classical theory of sketching is built on the notion of an oblivious subspace embedding (OSE). Fix integers $1 \le s \le n$ and $k\ge 1$, embedding parameters $\alpha \in (0,1]$ and $\beta\geq 1$, and a failure parameter $\rho \in [0,1)$. A random matrix $\Omega \in \R^{n \times k}$ is an $(s,\alpha,\beta,\rho)$-OSE if, for every $s$-dimensional subspace $V \subseteq \R^n$,
\[
\Prob\!\left\{
\alpha \norm{x}_2^2 \le \norm{\Omega^\top x}_2^2 \le \beta \norm{x}_2^2
\quad \text{for all } x \in V
\right\}
\ge 1-\rho.
\]
This property ensures that the sketch preserves the geometry of every low-dimensional subspace of the prescribed dimension. In particular, for sketch-and-solve least squares, one applies the OSE with $s=\dim\operatorname{span}(\range(A),b)\le d+1$, while for randomized low-rank approximation with target rank $r$, one applies it with $s=r+1$ to the augmented subspaces that arise in the analysis. See Sarl\'os~\cite{Sarlos2006}, Drineas--Mahoney--Muthukrishnan--Sarl\'os~\cite{DrineasMahoneyMuthukrishnanSarlos2011}, Rokhlin--Tygert~\cite{RokhlinTygert2008}, Clarkson--Woodruff~\cite{ClarksonWoodruff2013}, and Woodruff~\cite{Woodruff2014} for least squares, and Rokhlin--Szlam--Tygert~\cite{RokhlinSzlamTygert2009}, Halko--Martinsson--Tropp~\cite{HalkoMartinssonTropp2011}, Martinsson--Tropp~\cite{MartinssonTropp2020}, and Wang--Townsend~\cite{wang2026beyond} for randomized low-rank approximation. The resulting estimates take the form:
\[
\norm{A\widetilde{x} - b}_2 \leq \sqrt{\frac{\beta}{\alpha}}\norm{Ax_\star - b}_2, \qquad \|A - \widetilde{A}\|_F\leq \sqrt{\frac{\beta}{\alpha}}\norm{A - A_r}_F,
\]
where $A_r$ is the best rank-$r$ approximation to $A$, and $\alpha$ and $\beta$ are the OSE parameters.  In this paper, when we refer to OSE-style relative error, we mean a guarantee in which, when the sketch has near-isometric OSE parameters $\alpha = 1 - O(\eps)$, $\beta = 1 + O(\eps)$, and failure probability $\rho = O(\eps)$, the resulting approximation factor is $1 + O(\eps)$ and the overall failure probability is $O(\eps)$, with both depending only on the sketch parameters.

Cama\~no, Epperly, Meyer, and Tropp recently introduced a much weaker sketching property, called the oblivious subspace injection (OSI) property, and showed that it is sufficient for constant-factor guarantees in sketch-and-solve least squares and low-rank approximation via randomized SVD~\cite[Theorems~2.2 and~2.7]{CamanoEtAl2025}.  In other words, they proved that the sketched solutions were within a constant factor of optimal,  but their results do not identify a parameter regime in which that constant approaches $1$. 

\begin{definition}
Fix integers $1 \le s \le n$ and $k\ge 1$, an injectivity parameter $\alpha \in (0,1]$, and a failure parameter $\rho \in [0,1)$. A random matrix $\Omega \in \R^{n\times k}$ is an $(s,\alpha,\rho)$-OSI if:
\begin{enumerate}[label=(\roman*),leftmargin=1.8em]
    \item \textbf{Isotropy:} $\E[\Omega\Omega^\top] = I_n$.
    \item \textbf{Injectivity:} for every fixed $s$-dimensional subspace $V \subseteq \R^n$,
    \[
        \Prob\!\left\{ \norm{\Omega^\top x}_2^2 \ge \alpha\,\norm{x}_2^2
        \text{ for all } x \in V \right\} \ge 1-\rho.
    \]
\end{enumerate}
\label{def:OSI}
\end{definition}
\Cref{def:OSI} is very close to the one in~\cite{CamanoEtAl2025}. The only substantive difference is that we parameterize the failure probability by $\rho$, whereas their definition fixes the injectivity event to hold with probability at least $19/20$. Thus their $(s,\alpha)$-OSI is the special case of our $(s,\alpha,1/20)$-OSI.
For historical context, the open-problems note~\cite[\S5.1]{AmselEtAl2026} points to related one-sided conditions in earlier regression work, notably Drineas--Mahoney--Muthukrishnan--Sarl\'os~\cite{DrineasMahoneyMuthukrishnanSarlos2011}.

The OSI property is especially useful for structured random matrices, where the OSE property is difficult to guarantee.  The Cama\~no, Epperly, Meyer, and Tropp paper guarantees the OSI property for several fast sketches,  including sparse maps, subsampled randomized trigonometric transforms, and tensor-product constructions~\cite{CamanoEtAl2025}, whereas the previous literature had mostly focused on proving the OSE property with various degrees of success; see Nelson--Nguy$\tilde{\hat{\mathrm{e}}}$n~\cite{NelsonNguyen2013}, Ailon--Chazelle~\cite{AilonChazelle2009}, Tropp~\cite{Tropp2011}, and Pham--Pagh~\cite{PhamPagh2013}. More recently, tensor-train sketching has also been analyzed through the OSI lens~\cite{CazeauxDupuyJustiniano2026}.

In this paper, we show that OSI is not strong enough to imply OSE-style relative-error guarantees.  It turns out that the OSI property is significantly weaker than OSE (see~\cref{sec:OSI-implies-OSE}),  and that weakness is precisely what makes it possible to verify OSI for structured sketches in settings where proving a full OSE may be difficult or out of reach.  At the same time, our results show that this added flexibility comes at a price: OSI alone cannot guarantee relative error in the same way that OSE can when the failure probability is required to be controlled only by the OSI parameter. Indeed,  we find that relative-error bounds require upper control on the optimal residual in least squares or on the tail component in randomized SVD, whereas OSI provides only lower control on the range of $A$ together with isotropy in expectation.  Thus,  Cama\~no, Epperly, Meyer, and Tropp showed that constant-factor guarantees remain available far beyond the classical OSE setting~\cite{CamanoEtAl2025},  while our counterexamples show why there is still a genuine need for OSE-type upper control when relative-error bounds are the goal.

In practice,  however,  OSI-based sketches behave very similarly to OSE sketches (see~\cref{fig:osi-vs-ose-fixed-budget}), and often produce solutions of comparable quality.  This is consistent with the broader numerical experience with structured randomized least-squares solvers and low-rank sketching methods; see, e.g., Blendenpik~\cite{AvronMaymounkovToledo2010} and Tropp--Yurtsever--Udell--Cevher~\cite{TroppYurtseverUdellCevher2017}. The point of this paper is therefore not that OSI is ineffective in practice,  but rather that it is too weak to guarantee the same OSE-style relative error guarantees that are known under OSE assumptions (unless you assume a little bit more,  see~\cref{sec:positiveForLS,sec:positiveForrSVD}).

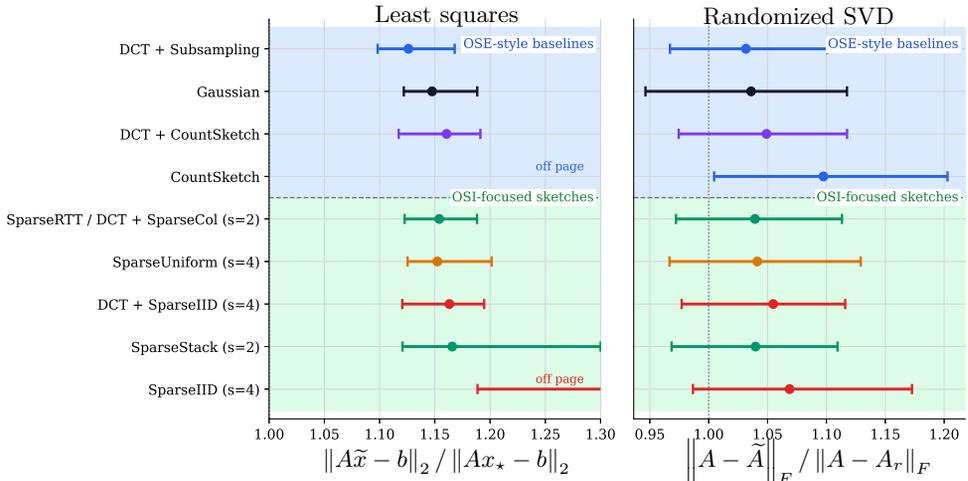
\begin{figure}[t]
\centering
\begin{overpic}[width=\linewidth]{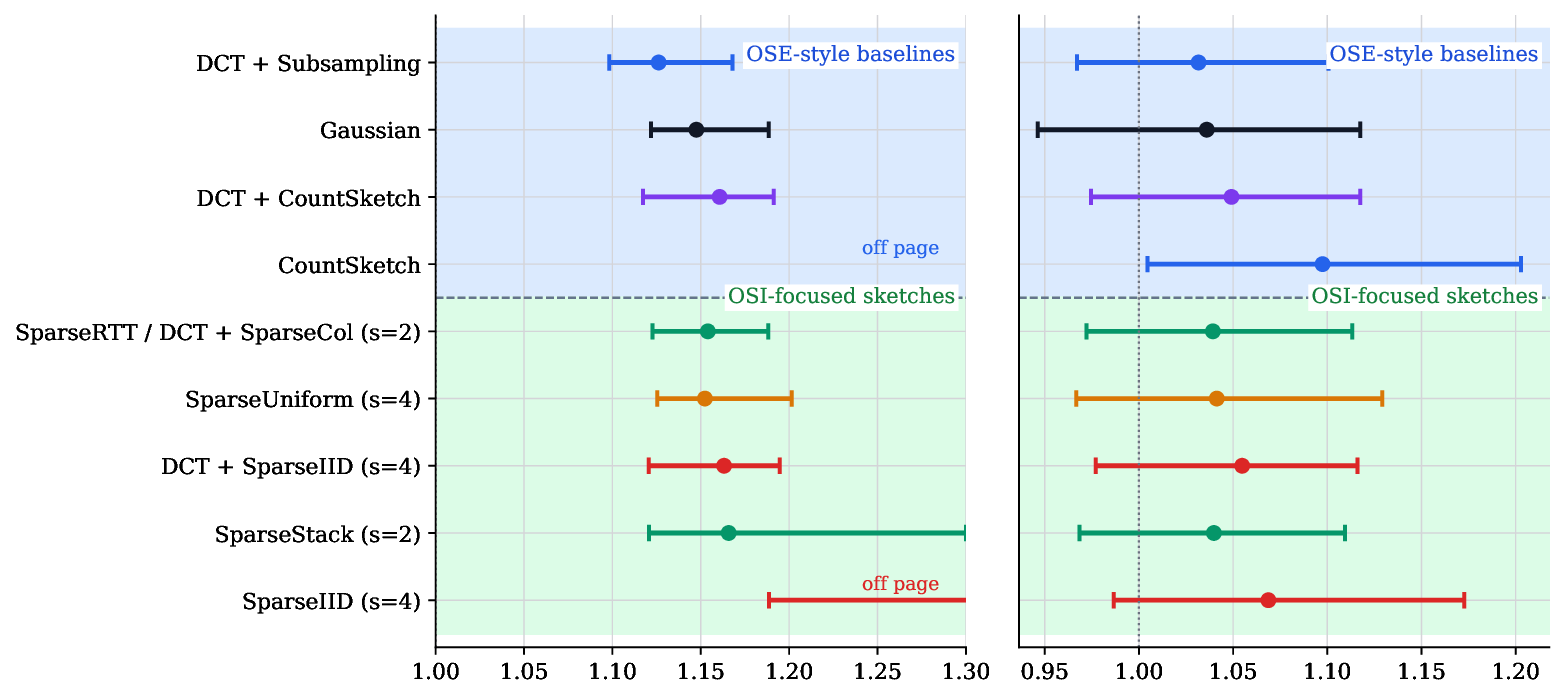}
\put(46,44){\makebox(0,0){\small Least squares}}
\put(82,44){\makebox(0,0){\small Randomized SVD}}
\put(46,-1){\makebox(0,0){\small $\norm{A\widetilde{x}-b}_2 / \norm{Ax_\star-b}_2$}}
\put(83,-1){\makebox(0,0){\small $\norm{A-\widetilde{A}}_F / \norm{A-A_r}_F$}}
\end{overpic}
\caption{Comparison of OSE and OSI sketches (see~\cite{CamanoEtAl2025} for details on the sketches), under a fixed sketching budget for regression and low-rank approximation,  where each dot marks the median over $100$ trials and each horizontal whisker marks the $10$th to $90$th percentile range.  Left: sketch-and-solve least squares with a matrix $A\in\R^{1024\times 64}$ with geometrically decaying singular values between $1$ and $0.12$, and $b=Ax_{\star}+e$, where $e\perp \range(A)$ is scaled so that $\norm{e}_2 = 0.2\,\norm{Ax_{\star}}_2/\sqrt{1024}$; the sketch dimension is $k=256$.  Right: randomized SVD on a matrix $A\in\R^{320\times 160}$ with target rank $r=10$ and exponentially decaying tail singular values,  using sketch size $k=r+p=15$ with oversampling $p=5$.  In both panels, the OSE methods are Gaussian or structured near-isometries, while the OSI methods are isotropic one-sided sketches.  In practice,  most OSE and OSI sketches give excellent relative approximations.}
\label{fig:osi-vs-ose-fixed-budget}
\end{figure}

The paper is organized as follows. In \Cref{sec:OSI-implies-OSE}, we show that OSI implies only a weak form of OSE, and that in the zero-failure regime $\rho=0$ the resulting upper-distortion bound is sharp.  In \Cref{sec:least-squares}, we prove that OSI does not imply OSE-style relative-error guarantees for sketch-and-solve least squares, even under stronger injectivity assumptions, and we also show how injectivity on the augmented space $\operatorname{span}(\range(A),b)$ recovers a relative bound. In \Cref{sec:rsvd}, we establish the analogous negative result for randomized SVD in the Frobenius norm and propose an extra assumption that recovers a relative-error guarantee. Finally, in \Cref{sec:lp}, we introduce a natural $\ell_p$ analogue of OSI and prove a constant-factor sketch-and-solve theorem for $\ell_p$ regression. 

\section{OSI Implies a Weak OSE}\label{sec:OSI-implies-OSE}

At first glance, OSI can seem less different from OSE than its definition suggests. Indeed, once one combines the injectivity part of OSI with isotropy, one finds that OSI does imply an OSE-type statement.  This raises a natural question: if OSI already implies an OSE-type statement, in what sense is OSI actually weaker? The answer is that the weakness lies in the parameters.  The upper-distortion bound obtained from isotropy is necessarily very coarse,  and the failure probability also deteriorates some.  Thus OSI implies only a weak OSE, not the near-isometric OSE needed for relative-error guarantees. The point of this section is to make that distinction precise and to show that, in the zero-failure regime $\rho=0$, the growth of the upper-distortion parameter is unavoidable. 

\begin{proposition}\label{prop:osi-markov-ose}
Let $\Omega\in\R^{n\times k}$ be an $(s,\alpha,\rho)$-OSI. Then, for every $0<\tau<1-\rho$, $\Omega$ is an $(s,\alpha,\alpha+s(1-\alpha+\alpha\rho)/\tau,\rho+\tau)$-OSE.
\end{proposition}

\begin{proof}
Fix an $s$-dimensional subspace $V\subseteq\R^n$, and let
$U\in\R^{n\times s}$ have orthonormal columns spanning $V$.
Set
$
G = U^\top \Omega\Omega^\top U \in \R^{s\times s}.
$
Then $G\succeq 0$, and for every $x=Uy\in V$,
$
\norm{\Omega^\top x}_2^2 = y^\top G y.
$ By isotropy,
$
\mathbb E[G]
= U^\top \mathbb E[\Omega\Omega^\top]U
= U^\top I_n U
= I_s,
$
and hence
\[
\mathbb E[\operatorname{tr}(G)] = \operatorname{tr}(I_s)=s.
\]
Let
$
E_V=\{G\succeq \alpha I_s\}.
$
By the injectivity part of the OSI hypothesis, $\Prob(E_V)\ge 1-\rho$.  On $E_V$, the matrix $G-\alpha I_s$ is positive semidefinite, so
\[
\lambda_{\max}(G)=
\alpha+\lambda_{\max}(G-\alpha I_s)
\le
\alpha+\operatorname{tr}(G-\alpha I_s).
\]
Also,
\[
\E\!\left[(\operatorname{tr}(G)-\alpha s)\mathbf 1_{E_V}\right]
\le
\E[\operatorname{tr}(G)]-\alpha s\,\Prob(E_V)
\le
s-\alpha s(1-\rho)
=
s(1-\alpha+\alpha\rho).
\]
Therefore, for any $\tau\in(0,1-\rho)$, Markov's inequality gives
\[
\Prob\!\left(
E_V
\cap
\left\{
\operatorname{tr}(G)-\alpha s>
\frac{s(1-\alpha+\alpha\rho)}{\tau}
\right\}
\right)
\le \tau.
\]
Hence, with probability at least $1-\rho-\tau$,  $E_V$ holds and
\[
\operatorname{tr}(G)-\alpha s
\le
\frac{s(1-\alpha+\alpha\rho)}{\tau},
\qquad
\lambda_{\max}(G)\le \alpha+\frac{s(1-\alpha+\alpha\rho)}{\tau}.
\]
On this event we have 
\[
\begin{aligned}
\alpha \norm{x}_2^2
=
\alpha \norm{y}_2^2
\le
y^\top G y
=
\norm{\Omega^\top x}_2^2
& \le
\left(\alpha+\frac{s(1-\alpha+\alpha\rho)}{\tau}\right)\norm{y}_2^2\\
& =
\left(\alpha+\frac{s(1-\alpha+\alpha\rho)}{\tau}\right)\norm{x}_2^2,
\end{aligned}
\]
for every $x=Uy\in V$. Since $V$ was arbitrary, this is exactly the claimed OSE property.
\end{proof}

\Cref{prop:osi-markov-ose} shows that an OSI gives an OSE only with poor upper-distortion parameters. When $\alpha=1-\eps$ and $\rho=0$, the proposition yields
\[
\beta = 1-\eps+\frac{s\eps}{\tau}.
\]
Thus the excess distortion $\beta-\alpha$ is of order $s(1-\alpha)/\tau$. Even when $\alpha$ is close to $1$, a near-isometric OSE still requires $s(1-\alpha)/\tau$ to be small, which is much stronger than OSI alone provides. One can nevertheless obtain weak relative-error guarantees for OSIs using existing OSE guarantees, which come with an additional loss of control over the failure probability of such guarantees. For instance, if $\Omega$ is an $(s,1-\eps,O(\eps))$-OSI, then by taking $\tau=\Theta(\eps^{1/2})$ we have that $\Omega$ is also an $(s,1-\eps,1+O(s\eps^{1/2}),\Theta(\eps^{1/2}))$-OSE, which satisfies relative-error guarantees of the form
\[
\|A\widetilde{x}-b\|_2 \le (1+O(\eps^{1/2}))\|Ax_\star-b\|_2, \qquad \|A-\widetilde{A}\|_F \le (1+O(\eps^{1/2}))\|A-A_r\|_F
\]
with failure probability $\Theta(\eps^{1/2})$.  However, as we show in \Cref{sec:least-squares,sec:rsvd}, OSI alone is not strong enough to obtain $1+O(\eps)$ relative error with $O(\eps)$ failure probability.

Since the isotropy property of a sketch in~\cref{def:OSI} fixes only the average trace on an $s$-dimensional subspace,  one can concentrate that excess trace into one direction.  \Cref{prop:osi-markov-sharp} shows that this pathology actually occurs and the dependence on $s(1-\alpha)/\tau$ in \Cref{prop:osi-markov-ose} is sharp when $\rho=0$.
 
\begin{proposition}\label{prop:osi-markov-sharp}
Fix $s\ge 1$, $\alpha\in(0,1)$, and $q\in(0,1)$.
Then there exists an $(s,\alpha,0)$-OSI $\Omega$ such that
\[
\mathbb P\!\left\{
\sup_{x\neq 0}
\frac{\norm{\Omega^\top x}_2^2}{\norm{x}_2^2}
=
\alpha+\frac{s(1-\alpha)}{q}
\right\}
= q.
\]
\end{proposition}

\begin{proof}
Take $n=s$. Let $J$ be uniform on $\{1,\dots,s\}$, let
$B\sim\mathrm{Bernoulli}(q)$, and assume $J$ and $B$ are independent.
Define the random positive semidefinite matrix
$
S
=
\alpha I_s
+
B\,\frac{s(1-\alpha)}{q}\, e_J e_J^\top,
$
and set
$
\Omega = S^{1/2}.
$
Then $\Omega\Omega^\top = S$.  First, $S\succeq \alpha I_s$ almost surely, so
\[
\norm{\Omega^\top x}_2^2 = x^\top S x \ge \alpha \norm{x}_2^2
\qquad x\in\R^s.
\]
Moreover,  we have
\[
\mathbb E[\Omega\Omega^\top]
=
\alpha I_s
+
q\cdot \frac{s(1-\alpha)}{q}\cdot
\mathbb E[e_J e_J^\top]
=
\alpha I_s
+
s(1-\alpha)\cdot \frac1s I_s
=
I_s,
\]
so $\Omega$ is an $(s,\alpha,0)$-OSI.  Finally, on the event $\{B=1\}$,
$
\Omega\Omega^\top =\alpha I_s+\frac{s(1-\alpha)}{q} e_J e_J^\top,
$
whose largest eigenvalue is
$
\lambda_{\max}(\Omega\Omega^\top)=\alpha+\frac{s(1-\alpha)}{q}.
$
This event has probability $q$, so
\[
\sup_{x\ne 0}\frac{\norm{\Omega^\top x}_2^2}{\norm{x}_2^2}
=
\alpha+\frac{s(1-\alpha)}{q}
\]
with probability exactly $q$.  
\end{proof}

Thus, OSI is substantially weaker than OSE and one cannot hope to obtain relative-error bounds for sketch-and-solve least squares or randomized SVD merely by passing from OSI to OSE.  Any such argument loses far too much in the parameters.

\section{Sketch-and-solve least squares with OSI}\label{sec:least-squares}
This section addresses the least-squares version of the main question: does OSI suffice to guarantee relative error for the sketch-and-solve estimator? This directly answers the open problem, labeled as Problem 5.1 in~\cite{AmselEtAl2026}.  

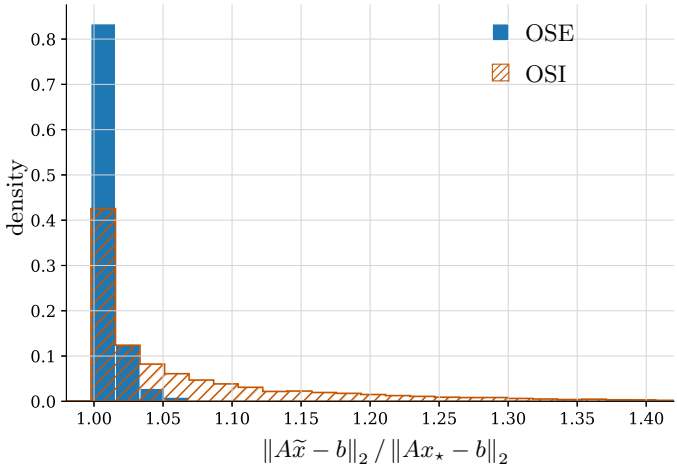
\begin{figure}[t]
\centering
\begin{overpic}[width=0.68\linewidth]{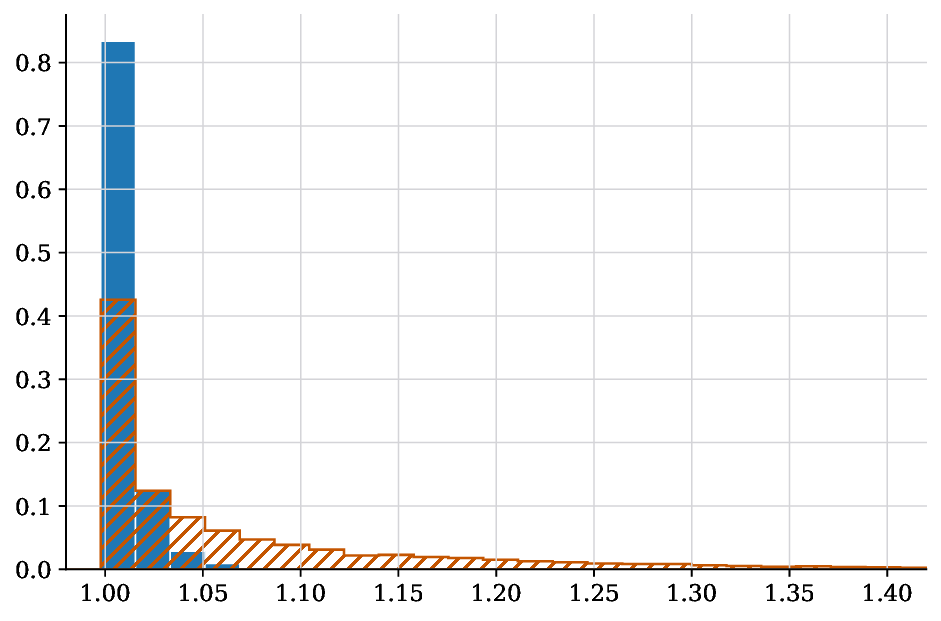}
\put(55,-2){\makebox(0,0){\small $\norm{A\widetilde x-b}_2 / \norm{A x_\star-b}_2$}}
\put(-2,30){\rotatebox{90}{\small density}}
\put(71,59){\tikz\fill[draw=white,fill={rgb,1:red,0.12;green,0.47;blue,0.71}] (0,0) rectangle (0.22,0.22);}
\put(76,59){\small OSE}
\put(71,53){\tikz\filldraw[pattern=north east lines,pattern color={rgb,1:red,0.77;green,0.33;blue,0.00},draw={rgb,1:red,0.77;green,0.33;blue,0.00}] (0,0) rectangle (0.22,0.22);}
\put(76,53){\small OSI}
\end{overpic}
\caption{Overlaid histograms for sketch-and-solve least squares on the toy problem $A = [1\ 0]^\top$ and $b = [0\ 1]^\top$ with 10,000 trials.  The blue histogram uses Gaussian sketches that behave as near-isometries, while the orange histogram uses an isotropic sketch family that satisfies $(1,\alpha,0)$-OSI with $\alpha=0.5$.  The OSI distribution causes a slightly heavier right tail.}
\label{fig:ls-osi-vs-ose-hist}
\end{figure}

In~\cref{fig:ls-osi-vs-ose-hist} we give a qualitative comparison between OSI and OSE sketches for regression.  The OSE sketch comes from a scaled Gaussian random matrix, which behaves like a near-isometry.  The OSI sketch comes from the continuous OSI family built from random positive semidefinite matrices
$
\Omega\Omega^\top = \alpha I_2 + Tuu^\top,
$
where $u$ is uniform on the unit circle and $T$ is an independent exponential random variable with mean $2(1-\alpha)$.  This sketch is isotropic and satisfies a $(1,\alpha,0)$-OSI,  but it has no comparable upper-distortion control. The OSI histogram has a slightly heavier right tail for the ratio $\norm{A\widetilde x-b}_2/\norm{A x_\star-b}_2$. 

\subsection{Counterexamples to relative accuracy for sketch-and-solve with OSI}

Since least-squares residuals live in the space $\operatorname{span}(\range(A),b)$,  one might initially hope that lower control on $\range(A)$ together with isotropy would already be enough to force the sketched problem to behave like the original one.  The results below show that this is false.  OSI controls the sketch on $\range(A)$ but it does not prevent the sketch from distorting the residual direction,  and that is exactly the direction that determines whether a relative-error guarantee can hold.

We begin with a minimal counterexample showing that even an $(1,1,\rho)$-OSI can incur a constant-factor loss on an event of probability $\rho$.  We then strengthen the construction substantially: even when the sketch is injective on every one-dimensional subspace,  so that the OSI failure probability is zero,  the sketch-and-solve estimator can still incur a constant-factor loss on an event of probability $\Omega(\eps)$. 

\begin{theorem}\label{thm:ls-counterexample}
For every $\rho \in (0,1)$, there exist a full-rank matrix $A\in\R^{2\times 1}$, a vector $b\in\R^2$, and an $(1,1,\rho)$-OSI $\Omega\in\R^{2\times 2}$ such that the sketch-and-solve estimator satisfies
\[
\frac{\norm{A\widetilde x-b}_2}{\min_{x\in\R} \norm{Ax-b}_2} = \begin{cases} 1, & \text{with probability $1-\rho$}, \\ \sqrt{2},& \text{with probability $\rho$}.
\end{cases} 
\]
\end{theorem}

\begin{proof}
Take
\[
A = \begin{bmatrix}1\\0\end{bmatrix},
\qquad
b = \begin{bmatrix}0\\1\end{bmatrix},
\qquad
B_+ = \begin{bmatrix}1&0\\1&0\end{bmatrix},
\qquad
B_- = \begin{bmatrix}1&0\\-1&0\end{bmatrix}.
\]
Define the random sketch by
\[
\Omega =
\begin{cases}
I_2, & \text{with probability } 1-\rho,\\[2mm]
B_+, & \text{with probability } \rho/2,\\[2mm]
B_-, & \text{with probability } \rho/2.
\end{cases}
\]
A direct computation gives
\[
\E[\Omega\Omega^\top]
=(1-\rho)I_2 + \frac{\rho}{2} B_+B_+^\top + \frac{\rho}{2} B_-B_-^\top
= I_2,
\]
so isotropy holds. For injectivity, it is clear that $\|\Omega^\top x\|_2^2=\|x\|_2^2$ with probability $1-\rho$, on the event $\{\Omega=I_2\}$, so $\Omega$ is a $(1,1,\rho)$-OSI.

Now
\[
\min_{x\in\R} \norm{Ax-b}_2^2 = \min_{x\in\R} (x^2+1)=1,
\qquad x_\star=0.
\]
If $\Omega=B_\pm$, then
\[
\Omega^\top A = \begin{bmatrix}1\\0\end{bmatrix},
\qquad
\Omega^\top b = \pm \begin{bmatrix}1\\0\end{bmatrix}.
\]
Hence the sketched least-squares problem is minimized at $\widetilde x = \pm 1$, and so
\[
A\widetilde x - b = \begin{bmatrix}\pm 1\\-1\end{bmatrix},
\qquad
\norm{A\widetilde x-b}_2^2 = 2.
\]
This bad event occurs with probability $\rho$.
\end{proof}

\Cref{thm:ls-counterexample} rules out OSE-style relative-error guarantees with controlled failure probability: in this example, the OSI fails to achieve relative-error guarantees with arbitrarily large probability $\rho$, regardless of how close the injectivity parameter is to 1.

Even a globally injective sketch can fail with probability $\Omega(\eps)$, as the next counterexample shows. 

\begin{theorem}\label{thm:ls-stronger}
For every $\eps\in (0,1)$ and $L\geq 1$, there exist $A\in\R^{2\times 1}$, $b\in\R^2$, and an $(1,1-\eps,0)$-OSI $\Omega\in\R^{2\times 3}$ such that the sketch-and-solve estimator satisfies
\[
\Prob\!\left\{ \norm{A\widetilde x-b}_2^2 \ge \left(1 + \frac{L^2}{(1+L)^2}\right) \min_{x\in\R} \norm{Ax-b}_2^2 \right\} \ge \frac{\eps}{2L}.
\]
\end{theorem}

\begin{proof}
Consider
\[
A = \begin{bmatrix}1\\0\end{bmatrix},
\qquad
b = \begin{bmatrix}0\\1\end{bmatrix}.
\]
Let $t = 2L/\eps>2$, and define a random vector $u\in\R^2$ by
\[
u_+ = \sqrt{\frac{t}{2}}\begin{bmatrix}1\\1\end{bmatrix},
\quad
u_- = \sqrt{\frac{t}{2(t-1)}}\begin{bmatrix}1\\-1\end{bmatrix},
\quad
u =
\begin{cases}
u_+, & \text{with probability } 1/t,\\[1mm]
u_-, & \text{with probability } 1-1/t.
\end{cases}
\]
Then
\[
\E[uu^\top]
= \frac{1}{t}\frac{t}{2}
\begin{bmatrix}1&1\\1&1\end{bmatrix}
+ \left(1-\frac{1}{t}\right)\frac{t}{2(t-1)}
\begin{bmatrix}1&-1\\-1&1\end{bmatrix}
= I_2.
\]
Now define
\[
\Omega = \bigl[\sqrt{1-\eps}\,I_2 \ \ \sqrt{\eps}\,u\bigr] \in \R^{2\times 3}.
\]
Because $\E[uu^\top]=I_2$, we have
\[
\E[\Omega\Omega^\top] = (1-\eps)I_2 + \eps\,\E[uu^\top] = I_2.
\]
Furthermore, for every $x\in\R^2$,
\[
\norm{\Omega^\top x}_2^2 = (1-\eps)\norm{x}_2^2 + \eps\,(u^\top x)^2 \ge (1-\eps)\norm{x}_2^2.
\]
Thus the injectivity inequality holds for \emph{every} vector, not just every one-dimensional subspace; in particular, $\Omega$ is an $(1,1-\eps,0)$-OSI.  Write $u=(g,s)^\top$. For any scalar $x$,
\[
Ax-b = \begin{bmatrix}x\\-1\end{bmatrix},
\qquad
\norm{\Omega^\top(Ax-b)}_2^2 = (1-\eps)(x^2+1) + \eps(gx-s)^2.
\]
Differentiating shows that the sketched objective is minimized at
\[
\widetilde x = \frac{\eps g s}{1-\eps + \eps g^2}.
\]
Hence
\[
\norm{A\widetilde x-b}_2^2 = 1+\widetilde x^2 = 1 + \frac{\eps^2 g^2 s^2}{(1-\eps+\eps g^2)^2}.
\]
On the event $\{u=u_+\}$ we have $g=s=\sqrt{t/2}$, and therefore
\[
\norm{A\widetilde x-b}_2^2
= 1 + \frac{(\eps t/2)^2}{(1-\eps+\eps t/2)^2} \ge 1+\frac{L^2}{(1+L)^2}
\]
since $\eps t/2= L$. Because $\min_x \norm{Ax-b}_2^2 = 1$, this gives
\[
\norm{A\widetilde x-b}_2^2 \ge \left(1 + \frac{L^2}{(1+L)^2}\right)\,\min_x \norm{Ax-b}_2^2
\qquad \text{on the event } \{u=u_+\}.
\]
Finally,
\[
\Prob\!\left\{u=u_+\right\} = \frac{1}{t} = \frac{\eps}{2L}.
\]
The final claim follows immediately. 
\end{proof}

\Cref{thm:ls-stronger} shows that even when the OSI failure probability is zero, one can still have a constant-factor loss on an event of probability $\Omega(\eps)$.  In particular, this rules out relative-error guarantees whose failure probability is controlled only by $\rho$, although it leaves open the possibility of such guarantees with fixed constant success probability or with failure probability of order $O(\eps)$.

The two counterexamples above exploit the same feature that relative-error regression depends on the geometry of $\operatorname{span}(\range(A),b)$,  but OSI controls only the lower singular behavior on $\range(A)$.  In other words, the sketch can leave $\range(A)$ intact while still distorting the optimal residual direction.  This points toward the natural remedy of imposing injectivity on the space $\operatorname{span}(\range(A),b)$.

\subsection{Injectivity on one extra dimension rescues relative error}\label{sec:positiveForLS}
When using the sketch-and-solve estimator on $Ax=b$,  we show that once a sketch is isotropic and injective on the subspace $\operatorname{span}(\range(A),b)$, the sketch-and-solve estimator satisfies a near-relative bound. 

Now,  injectivity on $\operatorname{span}(\range(A),b)$ controls the lower singular behavior of the sketch on all relevant directions, while isotropy controls the average size of the sketched optimal residual.  Together these are enough to rule out counterexamples.

\begin{proposition}
\label{prop:ls-p-one}
Let $A\in\R^{n\times d}$ be full rank,  $b\in\R^n$, and $U = \operatorname{span}(\range(A),b)$.  Suppose $\Omega$ satisfies $\E[\Omega\Omega^\top] = I_n$ and $\Prob\!\left\{
\norm{\Omega^\top u}_2^2 \ge \alpha \norm{u}_2^2\,\, \forall u \in U\right\}\ge 1 - \delta$.  Then,  for every $\eta\in(0,1)$, we have
\[
\Prob\!\left\{
\norm{A\widetilde x-b}_2^2
\le
\left(1+\frac{1-\alpha+\alpha\delta}{4\alpha\eta}\right)\norm{Ax_\star-b}_2^2
\right\}
\ge 1-\delta-\eta.
\]
If $\dim U<d+1$, extend $U$ to a $(d+1)$-dimensional subspace. Therefore, any $(d+1,\alpha,\delta)$-OSI sketch satisfies the assumptions of the proposition.
\end{proposition}
\begin{proof}
Let $A=QR$ be a reduced QR factorization, where $Q\in\R^{n\times d}$ has orthonormal
columns and $R\in\R^{d\times d}$ is invertible. Define
$
r_\star = (I-QQ^\top)b.
$
If $\norm{r_\star}_2=0$, then $b\in\range(A)$ and the claim is trivial. Hence assume $\norm{r_\star}_2>0$ and set
$
y = r_\star/\norm{r_\star}_2
$
Then $y\perp\range(Q)$ and
\[
b = Qz_\star + \norm{r_\star}_2 y,
\qquad
z_\star = Q^\top b.
\]
Because $R$ is invertible, minimizing $\norm{\Omega^\top(Ax-b)}_2$ over $x\in\R^d$ is
equivalent to minimizing
$
\norm{\Omega^\top(Qz-b)}_2
$
over $z\in\R^d$. Let $\widetilde z$ be a minimizer and write
$
e = \widetilde z-z_\star.
$
Then $\widetilde x = R^{-1}\widetilde z$ and
$
A\widetilde x-b = Qe-\norm{r_\star}_2 y.
$
Since $\widetilde z$ minimizes the sketched objective, the normal equations give
\[
Q^\top\Omega\Omega^\top(Qe-\norm{r_\star}_2 y)=0.
\]
Define
$
M = Q^\top\Omega\Omega^\top Q,$ 
$
g = Q^\top\Omega\Omega^\top y,$
$
t = y^\top\Omega\Omega^\top y.$
Then
$
Me = \norm{r_\star}_2 g.
$
Let $E$ denote the event that
\[
\norm{\Omega^\top u}_2^2 \ge \alpha\norm{u}_2^2
\qquad \text{for all } u\in U, 
\]
where $U = \operatorname{span}(\range(A),b) = \operatorname{span}(\range(Q),y)$.  On $E$, for every $v\in\R^d$, we have
$
v^\top Mv
=
\norm{\Omega^\top Qv}_2^2
\ge\alpha\norm{Qv}_2^2
=
\alpha\norm{v}_2^2.
$
Thus $M\succeq \alpha I_d$, so $M$ is invertible and
$
e = \norm{r_\star}_2 M^{-1}g.
$
Since $Qe\perp y$, we obtain
$
\norm{A\widetilde x-b}_2^2
=
\norm{Qe-\norm{r_\star}_2 y}_2^2
=
\norm{r_\star}_2^2+\norm{e}_2^2
=
\norm{r_\star}_2^2\bigl(1+g^\top M^{-2}g\bigr).
$
Now consider the block Gram matrix
\[
G =
\begin{bmatrix}
M & g\\
g^\top & t
\end{bmatrix}
=
\begin{bmatrix}Q & y\end{bmatrix}^\top
\Omega\Omega^\top
\begin{bmatrix}Q & y\end{bmatrix}.
\]
Because the columns of $[Q\ \ y]$ are orthonormal and $U=\operatorname{span}(\range(A),b)=\operatorname{span}(Q,y)$,
the event $E$ implies
$
G \succeq \alpha I_{d+1}.
$
Therefore
\[
\begin{bmatrix}
M- \alpha I_d & g\\
g^\top & t- \alpha 
\end{bmatrix}
\succeq 0.
\]
If $t= \alpha $, then positive definiteness forces $g=0$, and the desired bound is immediate. Otherwise,
the Schur complement gives
$
gg^\top \preceq \bigl(t- \alpha \bigr)\bigl(M- \alpha I_d\bigr).
$
Hence
\[
g^\top M^{-2}g
=
\lambda_{\max}\!\bigl(M^{-1}gg^\top M^{-1}\bigr)
\le
\bigl(t- \alpha \bigr)
\lambda_{\max}\!\bigl(M^{-1}(M- \alpha I_d)M^{-1}\bigr).
\]
If $\lambda\ge  \alpha $ is an eigenvalue of $M$, then the corresponding eigenvalue of
$M^{-1}(M- \alpha I_d)M^{-1}$ equals
$
(\lambda- \alpha )/\lambda^2,
$
which is maximized at $\lambda=2 \alpha $ with value $1/(4 \alpha )$. Therefore,
on $E$,
\[
g^\top M^{-2}g \le \frac{t- \alpha }{4 \alpha }.
\]
Substituting into the residual identity yields
\[
\norm{A\widetilde x-b}_2^2
\le
\left(1+\frac{t- \alpha }{4 \alpha }\right)\norm{r_\star}_2^2
\qquad \text{on } E.
\]
It remains to control $t$. Since $\norm{y}_2=1$ and $\E[\Omega\Omega^\top]=I_n$,
$
\E[t] = \E[y^\top\Omega\Omega^\top y] = 1.
$
Also, on $E$, we have $t\ge  \alpha $. Therefore
\[
\E\bigl[(t- \alpha )\mathbf 1_E\bigr]
=
\E[t\mathbf 1_E] -  \alpha \Prob(E)
\le
1- \alpha (1-\delta)
=
1- \alpha + \alpha \delta.
\]
By Markov's inequality, for every $\eta\in(0,1)$,
\[
\Prob\!\left(
E \cap
\left\{
t- \alpha  > \frac{1- \alpha + \alpha \delta}{\eta}
\right\}
\right)
\le \eta.
\]
Hence, with probability at least $1-\delta-\eta$,  $E$ holds and
$
t- \alpha  \le (1- \alpha + \alpha \delta)/\eta.
$
On this event,
\[
\norm{A\widetilde x-b}_2^2
\le
\left(1+\frac{1- \alpha + \alpha \delta}{4 \alpha \eta}\right)\norm{r_\star}_2^2.
\]
This proves the claim. 
\end{proof}

It is important to note that~\cref{prop:ls-p-one} assumes more than a $d$-dimensional OSI as we ask for injectivity on a space of dimension up to $d+1$.  Accordingly,~\cref{thm:ls-counterexample} does not contradict~\cref{prop:ls-p-one} as the sketch in the counterexample preserves $\range(A)$ but distorts the additional residual direction. 

If one takes $\delta=0$ and $\alpha=1-\eps$ in~\cref{prop:ls-p-one}, then we obtain
\[
\norm{A\widetilde x-b}_2^2
\le
\bigl(1+O(\eps/\eta)\bigr)\norm{Ax_\star-b}_2^2
\]
with probability at least $1-\eta$. Thus, for any fixed $\eta$, one obtains a $1+O(\eps)$ relative-error bound with constant success probability. However, if one insists on success probability tending to $1$ as $\eps\to 0$, then $\eta$ must also shrink, and the bound degrades accordingly. This leaves room for the $\Omega(\eps)$-probability constant-factor failure event exhibited in~\cref{thm:ls-stronger}.

\section{Randomized SVD with OSI}\label{sec:rsvd}

We now turn to the low-rank approximation analogue of the previous section and resolve Problem~5.2 from~\cite{AmselEtAl2026}. In the classical analysis of randomized SVD, relative Frobenius-error bounds come from controlling not only the leading right singular space, but also how the sketch interacts with the tail singular directions.  This is exactly why OSE is a natural sketching property: it provides two-sided control on the relevant subspaces and thereby bounds the interaction between the dominant and trailing components.

\begin{figure}[t]
\centering
\begin{overpic}[width=0.68\linewidth]{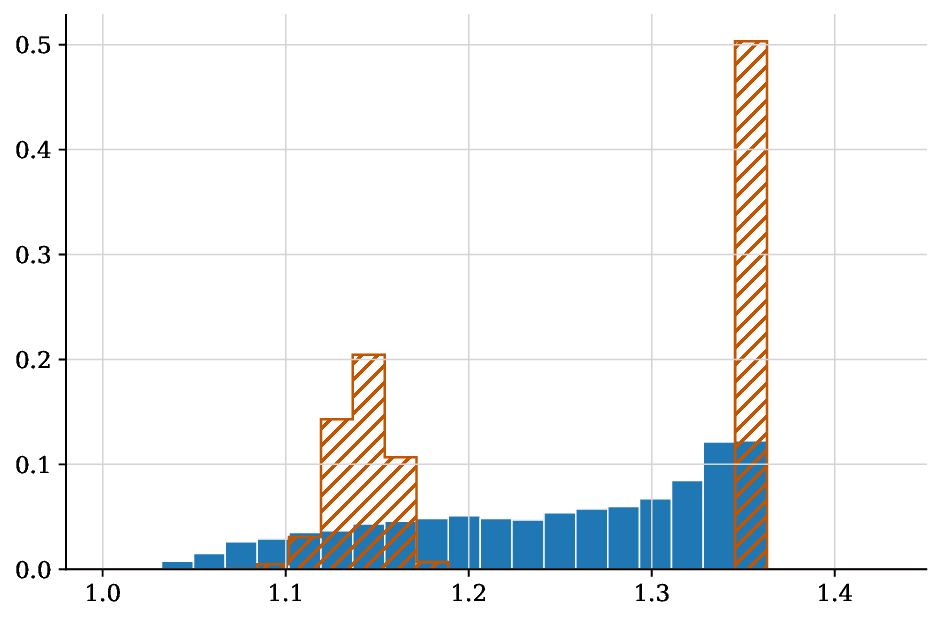}
\put(50,0){\makebox(0,0){\small $\|A-\widetilde A\|_F / \norm{A-A_1}_F$}}
\put(-2,27){\rotatebox{90}{\small density}}
\put(51,59){\tikz\fill[draw=white,fill={rgb,1:red,0.12;green,0.47;blue,0.71}] (0,0) rectangle (0.22,0.22);}
\put(56,59){\small OSE}
\put(51,53){\tikz\filldraw[pattern=north east lines,pattern color={rgb,1:red,0.77;green,0.33;blue,0.00},draw={rgb,1:red,0.77;green,0.33;blue,0.00}] (0,0) rectangle (0.22,0.22);}
\put(56,53){\small OSI}
\end{overpic}
\caption{Overlaid histograms for randomized SVD on the diagonal test matrix
$A=\diag(1,\tau,\ldots,\tau)\in\R^{30\times 30}$ with $\tau=0.2$, target rank $r=1$, and with 10,000 trials.  The blue histogram uses a single Gaussian sketch vector, while the orange histogram uses a sparse signed isotropic single-vector sketch with independent ternary entries.  Since the first coordinate of the OSI sketch vector is $0$ half the time, the dominate singular value is missed, causing a bimodal distribution. }
\label{fig:rsvd-osi-vs-ose-hist}
\end{figure}
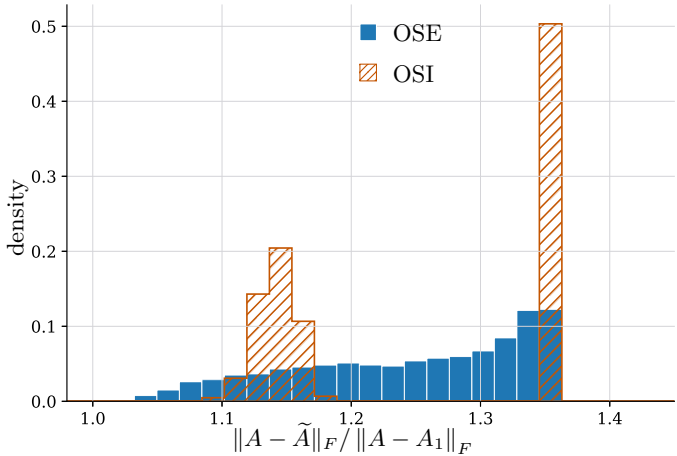

In~\cref{fig:rsvd-osi-vs-ose-hist}, we show a histogram of the ratio $\|A-\widetilde A\|_F / \norm{A-A_1}_F$ obtained from randomized SVD on a larger diagonal matrix.  The blue histogram uses a single Gaussian sketch vector, whereas the orange histogram uses a sparse signed sketch vector whose coordinates are independent and take the values $\pm\sqrt{2}$ with probability $1/4$ each and $0$ with probability $1/2$.  This one-vector family is isotropic, but it has no comparable upper-distortion control.  The sparsity of the sparse signed sketch vector causes the histogram to be bimodal. 

\subsection{Counterexample to relative accuracy for randomized SVD with OSI}
Unfortunately,  even in the smallest nontrivial example,  an OSI can produce a randomized SVD approximation whose Frobenius error is separated from the optimum by a fixed constant factor.  Thus, just as in least squares, OSI does not provide enough control for relative accuracy.

\begin{theorem}\label{thm:rsvd-counterexample}
For every $\tau\in(0,1)$ there exist a matrix $A\in\R^{2\times 2}$ with target rank $r=1$ and an $(1,1,1/2)$-OSI $\Omega\in\R^{2\times 1}$ such that the randomized SVD approximation satisfies
\[
\|A-\widetilde A\|_F
= 
\sqrt{\frac{2}{1+\tau^2}}\min_{\rank(B)\le 1}\|A-B\|_F, \qquad \text{almost surely.}
\]
In particular, the ratio tends to $\sqrt{2}$ as $\tau\to 0$.
\end{theorem}

\begin{proof}
Consider
\[
A = \diag(1,\tau),
\quad
\omega_+ = \begin{bmatrix}1\\1\end{bmatrix},
\quad
\omega_- = \begin{bmatrix}1\\-1\end{bmatrix},
\quad
\Omega=
\begin{cases}
\omega_+, & \text{with probability } 1/2,\\
\omega_-, & \text{with probability } 1/2.
\end{cases}
\]
Then
$
\E[\Omega\Omega^\top] = \tfrac12\omega_+\omega_+^\top + \tfrac12\omega_-\omega_-^\top = I_2,
$
so isotropy holds. To verify injectivity, fix any one-dimensional subspace $V=\operatorname{span}\{(x,y)^\top\}$. Since $(x+y)^2 + (x-y)^2 = 2(x^2+y^2)$, at least one of the two quantities $(x+y)^2$ and $(x-y)^2$ is at least $x^2+y^2$. Therefore, with probability at least $1/2$,
\[
\norm{\Omega^\top v}_2^2 \ge \norm{v}_2^2
\qquad \text{for every } v\in V,
\]
so $\Omega$ is an $(1,1,1/2)$-OSI.  Now fix either sign. Then
\[
Y = A\Omega = \begin{bmatrix}1\\ \pm \tau\end{bmatrix},
\qquad
\widetilde A = Y(Y^\top Y)^{-1}Y^\top A.
\]
Because $Y^\top Y = 1+\tau^2$, the orthogonal projector onto $\operatorname{span}(Y)$ is
\[
P_Y = \frac{1}{1+\tau^2}
\begin{bmatrix}
1 & \pm \tau\\
\pm \tau & \tau^2
\end{bmatrix},
\qquad
\widetilde A = P_YA.
\]
A direct computation yields
\[
A-\widetilde A
=
\frac{1}{1+\tau^2}
\begin{bmatrix}
\tau^2 & \mp \tau^2\\
\mp \tau & \tau
\end{bmatrix}.
\]
Hence
\[
\|A-\widetilde A\|_F^2
=
\frac{2\tau^4+2\tau^2}{(1+\tau^2)^2}
=
\frac{2\tau^2}{1+\tau^2}.
\]
On the other hand, the best rank-one approximation error is
\[
\min_{\rank(B)\le 1}\norm{A-B}_F^2 = \tau^2,
\]
by Eckart--Young.  The result follows.
\end{proof}

\Cref{thm:rsvd-counterexample} shows that OSI is too weak to guarantee relative Frobenius-error bounds for randomized SVD. The issue is analogous to the one in least squares,  but it appears here in a slightly different way.  OSI controls the sketch on the leading singular space,  yet it does not sufficiently control how the sketch mixes that space with the trailing singular directions.  As a result, the sketch can be injective where it needs to be while still producing a subspace that is misaligned enough to incur a fixed constant-factor loss.

This is also consistent with the standard deterministic analyses of randomized SVD~\cite{HalkoMartinssonTropp2011,MartinssonTropp2020}, which typically study quantities such as
\[
\Sigma_2(V_2^\top\Omega)(V_1^\top\Omega)^+
\]
where $A = U_1\Sigma_1V_1^\top + U_2\Sigma_2V_2^\top$ is a partitioned reduced SVD of $A$ (see~\cref{prop:rsvd-r-plus-one}).  Controlling this term requires more than one-sided injectivity on the dominant space.  \Cref{thm:rsvd-counterexample} shows that OSI leaves enough freedom in the sketch for this term to remain large, and hence for a relative-error guarantee to fail. In particular,  one cannot hope to prove relative accuracy for randomized SVD from OSI alone.

\subsection{Injectivity on the leading space plus each tail direction rescues relative error}\label{sec:positiveForrSVD}

The proof of~\cref{thm:rsvd-counterexample} also suggests what is missing from OSI.  For randomized SVD, it is not enough for the sketch to be injective on the dominant right singular space $\range(V_1)$; one must also control how each trailing singular direction interacts with that space.  The natural augmented subspaces are therefore
\[
W_j = \operatorname{span}(V_1,v_j),
\qquad j=r+1,\dots,q,
\]
where the columns of $V_1$ are the dominant $r$ right singular vectors of $A$ and $v_{r+1},\dots,v_q$ are the trailing right singular vectors.  Once injectivity holds simultaneously on these $(r+1)$-dimensional subspaces, isotropy again supplies the missing upper control in expectation, and a Markov argument yields a near-relative bound.

\begin{proposition}\label{prop:rsvd-r-plus-one}
Let $A\in\R^{n\times d}$ have rank $q>r$, and write a reduced singular value decomposition as
\[
A = U_1\Sigma_1V_1^\top + U_2\Sigma_2V_2^\top,
\]
where $\Sigma_1\in\R^{r\times r}$ contains the top $r$ singular values,
$
V_2 = [v_{r+1}\ \cdots\ v_q],
$
and
$
A_r = U_1\Sigma_1V_1^\top
$
is the best rank-$r$ approximation to $A$.  Let $\Omega\in\R^{d\times k}$, $k>r$, and set
$
\widetilde A = (A\Omega)(A\Omega)^+A.
$
Let $W_j=\operatorname{span}(V_1,v_j)$, $j=r+1,\dots,q$. Suppose $\E[\Omega\Omega^\top]=I_d$ and that, with probability at least $1-\delta$, $\norm{\Omega^\top x}_2^2 \ge \alpha \norm{x}_2^2$ holds for all $x\in W_j$ and every $j=r+1,\dots,q$. Then, for every $\eta\in(0,1)$,
\[
\Prob\!\left\{
\|A-\widetilde A\|_F^2
\le
\left(1+\frac{1-\alpha+\alpha\delta}{4\alpha\eta}\right)\norm{A-A_r}_F^2
\right\}
\ge 1-\delta-\eta.
\]
\end{proposition}
\begin{proof}

Set $\Omega_1 = V_1^\top\Omega \in \R^{r\times k}$ and $\Omega_2 = V_2^\top\Omega \in \R^{(q-r)\times k}$. From \cite[Theorem 9.1]{HalkoMartinssonTropp2011}, we have
\[
\|A-\widetilde A\|_F^2
\le
\norm{\Sigma_2}_F^2
+
\norm{\Sigma_2\Omega_2\Omega_1^+}_F^2.
\]
For each $j=r+1,\dots,q$, define
$
\omega_j^\top = v_j^\top\Omega \in \R^{1\times k},
$
so that the rows of $\Omega_2$ are $\omega_{r+1}^\top,\dots,\omega_q^\top$.  Then
\[
\norm{\Sigma_2\Omega_2\Omega_1^+}_F^2
=
\sum_{j=r+1}^q \sigma_j^2 \norm{\omega_j^\top\Omega_1^+}_2^2,
\]
where $\sigma_j$ is the $j$th singular value of $A$. 

We next bound the quantity on the right term-by-term. Let $E$ denote the event that
\[
\norm{\Omega^\top x}_2^2 \ge \alpha \norm{x}_2^2
\qquad \text{for all } x\in W_j,\ \ j=r+1,\dots,q.
\] Fix $j>r$, and set
$
M = V_1^\top\Omega\Omega^\top V_1 = \Omega_1\Omega_1^\top$, $g_j = V_1^\top\Omega\Omega^\top v_j = \Omega_1\omega_j$, and $ t_j = v_j^\top\Omega\Omega^\top v_j = \norm{\omega_j}_2^2$.  Consider the block Gram matrix
\[
G_j
=
\begin{bmatrix}
M & g_j\\
g_j^\top & t_j
\end{bmatrix}
=
\begin{bmatrix}
V_1 & v_j
\end{bmatrix}^\top
\Omega\Omega^\top
\begin{bmatrix}
V_1 & v_j
\end{bmatrix}.
\]
Because the columns of $[V_1\ \ v_j]$ are orthonormal and
$
W_j = \operatorname{span}(V_1,v_j),
$
the event $E$ implies
$
G_j \succeq \alpha I_{r+1},
$
hence
\[
\begin{bmatrix}
M-\alpha I_r & g_j\\
g_j^\top & t_j-\alpha
\end{bmatrix}
\succeq 0.
\]
If $t_j=\alpha$, then positivity forces $g_j=0$.  Otherwise, the Schur complement gives
\[
g_jg_j^\top \preceq (t_j-\alpha)(M-\alpha I_r).
\]
Since $\Omega_1^+=\Omega_1^\top(\Omega_1\Omega_1^\top)^{-1}=\Omega_1^\top M^{-1}$ on $E$, we have
\[
\norm{\omega_j^\top\Omega_1^+}_2^2
=
g_j^\top M^{-2}g_j
\le
(t_j-\alpha)\,
\lambda_{\max}\!\bigl(M^{-1}(M-\alpha I_r)M^{-1}\bigr).
\]
If $\lambda\ge \alpha$ is an eigenvalue of $M$, then the corresponding eigenvalue of
$
M^{-1}(M-\alpha I_r)M^{-1}
$
equals
$
(\lambda-\alpha)/\lambda^2,
$
which is maximized at $\lambda=2\alpha$ and has value $1/(4\alpha)$.  Therefore,
\[
\norm{\omega_j^\top\Omega_1^+}_2^2 \le \frac{t_j-\alpha}{4\alpha}
\qquad \text{on } E.
\]
Summing over $j$ yields
\[
\norm{\Sigma_2\Omega_2\Omega_1^+}_F^2
\le
\frac{1}{4\alpha}\sum_{j=r+1}^q \sigma_j^2(t_j-\alpha) = \frac{T-\alpha}{4\alpha}\,\norm{\Sigma_2}_F^2,\qquad T
=
\frac{1}{\norm{\Sigma_2}_F^2}
\sum_{j=r+1}^q \sigma_j^2 t_j,
\]
hence
\[
\|A-\widetilde A\|_F^2
\le
\left(1+\frac{T-\alpha}{4\alpha}\right)\norm{\Sigma_2}_F^2
\qquad \text{on } E.
\]

It remains to control the random variable $T$.  By isotropy,
$
\E[t_j]
=
\E[v_j^\top\Omega\Omega^\top v_j]
=
1$ for each $j=r+1,\dots,q$,  so
\[
\E[T]
=
\frac{1}{\norm{\Sigma_2}_F^2}
\sum_{j=r+1}^q \sigma_j^2 \E[t_j]
=
1.
\]
Also, on $E$, since $v_j\in W_j$, we have $t_j\ge \alpha$ for every $j>r$, and thus $T\ge \alpha$. Therefore
\[
\E\bigl[(T-\alpha)\mathbf 1_E\bigr]
=
\E[T\mathbf 1_E]-\alpha\Prob(E)
\le
1-\alpha(1-\delta)
=
1-\alpha+\alpha\delta.
\]
By Markov's inequality, for every $\eta\in(0,1)$,
\[
\Prob\!\left(
E \cap
\left\{
T-\alpha > \frac{1-\alpha+\alpha\delta}{\eta}
\right\}
\right)
\le \eta.
\]
Hence, with probability at least $1-\delta-\eta$,  $E$ holds and
$
T-\alpha \le (1-\alpha+\alpha\delta)/\eta.
$
On this event,
\[
\|A-\widetilde A\|_F^2
\le
\left(1+\frac{1-\alpha+\alpha\delta}{4\alpha\eta}\right)\norm{\Sigma_2}_F^2.
\]
The result follows from Eckart--Young as 
$
\norm{\Sigma_2}_F = \norm{A-A_r}_F.
$
\end{proof}

\Cref{prop:rsvd-r-plus-one} shows that injectivity on each augmented subspace $W_j$ rules out the counterexample in \Cref{thm:rsvd-counterexample}. In this sense, the missing ingredient for relative Frobenius error is again upper control on the tail, and isotropy supplies that upper control only after the relevant lower injectivity has been imposed.

\begin{corollary}\label{cor:rsvd-r-plus-one}
Let $A\in\R^{n\times d}$ have rank $q>r$, and suppose that $(q-r)\rho<1$.  If $\Omega\in\R^{d\times k}$ is an $(r+1,\alpha,\rho)$-OSI, then, for every $\eta\in(0,1)$,
\[
\Prob\!\left\{
\|A-\widetilde A\|_F^2
\le
\left(1+\frac{1-\alpha+\alpha(q-r)\rho}{4\alpha\eta}\right)\norm{A-A_r}_F^2
\right\}
\ge
1-(q-r)\rho-\eta.
\]
\end{corollary}
\begin{proof}
For each $j=r+1,\dots,q$, the space $W_j$ is $(r+1)$-dimensional.  By the OSI hypothesis,
\[
\Prob\!\left\{
\norm{\Omega^\top x}_2^2 \ge \alpha \norm{x}_2^2
\text{ for all } x\in W_j
\right\}
\ge 1-\rho.
\]
A union bound shows that the simultaneous event in \Cref{prop:rsvd-r-plus-one} holds with probability at least
$
1-(q-r)\rho.
$
Applying \Cref{prop:rsvd-r-plus-one} with $\delta=(q-r)\rho$ gives the result.
\end{proof}

If $\rho=0$ and $\alpha=1-\eps$ in \Cref{cor:rsvd-r-plus-one}, then for every fixed $\eta\in(0,1)$,
\[
\|A-\widetilde A\|_F^2
\le
\bigl(1+O(\eps/\eta)\bigr)\norm{A-A_r}_F^2
\]
with probability at least $1-\eta$.  Thus, just as in least squares, a near-relative-error bound is recovered once the sketch is injective on the relevant augmented spaces.

\section{An OSI analogue for $\ell_p$ regression}\label{sec:lp}
There is one more open problem in~\cite{AmselEtAl2026} regarding OSI,  labeled as Problem~5.3, and for completeness we answer it here.  Problem 5.3 is about solving $\ell_p$ regression using sketching,  so throughout we fix $1\le p<\infty$.  Here, one wishes to solve the $\ell_p$ regression problem, i.e., 
\[
x_\star \in \argmin_{x\in\R^d} \norm{A x - b}_p,
\]
where $\norm{\cdot}_p$ is the $\ell_p$ norm.  Instead of solving this directly, the sketch-and-solve estimator takes a sketch $\Omega\in\R^{n\times k}$ and gives 
\[
\widetilde x \in \argmin_{x\in\R^d} \norm{\Omega^\top(A x - b)}_p. 
\]

To formulate an OSI-type property of a sketch in this setting, it is helpful to recall what isotropy means in the Euclidean case. When $p=2$, the isotropy condition $\E[\Omega\Omega^\top]=I_n$ says that, on average, the sketch preserves squared Euclidean norm, i.e., 
\[
\E\!\left[\norm{\Omega^\top z}_2^2\right]=\norm{z}_2^2
\qquad \text{for every } z\in\R^n.
\]
For $\ell_p$ regression, the natural analogue is therefore to require preservation of the $p$th power of the $\ell_p$ norm in expectation, that is,
\[
\E\!\left[\norm{\Omega^\top z}_p^p\right]=\norm{z}_p^p.
\] 
This leads to a natural $\ell_p$ version of OSI, and below we show that it yields a constant-factor guarantee for sketch-and-solve $\ell_p$ regression.  

\begin{definition}
Given $s\in\mathbb{N}$, $\alpha\in(0,1]$, and $\rho\in[0,1)$,  we say that $\Omega\in\R^{n\times k}$ is an $(s,\alpha,\rho)$-OSI${}_p$ sketch if it satisfies: 
\begin{enumerate}[label=(\roman*),leftmargin=1.8em]
    \item \textbf{$\mathbf{p}$-isotropy:} $\E\left[\norm{\Omega^\top z}_p^p\right] = \norm{z}_p^p$ for every $z\in\R^n.$
    \item \textbf{Injectivity:} for every fixed $s$-dimensional subspace $V\subseteq\R^n$,
\[
\Prob\!\left\{ \norm{\Omega^\top v}_p^p \ge \alpha\,\norm{v}_p^p \text{ for all } v\in V \right\} \ge 1-\rho.
\]
\end{enumerate}
\end{definition}

The definition of $p$-isotropy is nonvacuous; for example, it is satisfied by the following simple sampling sketch. Fix integers $n,k\ge1$. Let $i_1,\dots,i_k$ be i.i.d.\ uniform random variables on $\{1,\dots,n\}$, and define
\[
\Omega
=
\left[
\left(\frac{n}{k}\right)^{1/p} e_{i_1}\ \cdots\ \left(\frac{n}{k}\right)^{1/p} e_{i_k}
\right]
\in \R^{n\times k}.
\]
Then $\Omega$ is $p$-isotropic, because for every $z\in\R^n$,
\[
\E\!\left[\norm{\Omega^\top z}_p^p\right]
=
\sum_{j=1}^k \E\!\left[\frac{n}{k}|z_{i_j}|^p\right]
=
\sum_{j=1}^k \frac{1}{k}\sum_{\ell=1}^n |z_\ell|^p
=
\norm{z}_p^p.
\]
However, constructing practical sketches that satisfy both $p$-isotropy and the corresponding injectivity property is a more subtle matter. This question is closely related to the literature on $\ell_p$ subspace embeddings and robust regression; see, for example, Sohler--Woodruff~\cite{SohlerWoodruff2011}, Meng--Mahoney~\cite{MengMahoney2013}, Woodruff--Zhang~\cite{WoodruffZhang2013}, Cohen--Peng~\cite{CohenPeng2015}, and the survey by Woodruff~\cite{Woodruff2014}.  We do not have any new ideas on this.

We begin with a deterministic statement, which is the basic mechanism behind the probabilistic result. 

\begin{theorem}\label{thm:deterministic-lp}
Let $A\in\R^{n\times d}$ and $b\in\R^n$.  Choose any optimal solution $x_\star\in\argmin_x \norm{Ax-b}_p$, and write $r_\star = Ax_\star-b$. Suppose that the sketch $\Omega\in\R^{n\times k}$ satisfies:
\begin{enumerate}[label=(\roman*),leftmargin=1.8em]
    \item $\norm{\Omega^\top v}_p^p \ge \alpha\,\norm{v}_p^p$ for every $v\in\range(A)$;
    \item $\norm{\Omega^\top r_\star}_p^p \le \beta\,\norm{r_\star}_p^p$.
\end{enumerate}
Then, 
\[
\norm{A\widetilde x-b}_p \le \left(1 + 2\left(\frac{\beta}{\alpha}\right)^{1/p}\right)\min_{x\in\R^d} \norm{Ax-b}_p.
\]
\end{theorem}

\begin{proof}
Set $u=A(\widetilde x-x_\star)\in\range(A)$. Then
$
A\widetilde x-b = r_\star + u.
$
Since $\widetilde x$ minimizes the sketched objective,
\[
\norm{\Omega^\top(r_\star+u)}_p \le \norm{\Omega^\top r_\star}_p.
\]
Taking $p$th roots in assumptions~(i) and~(ii), and then using the triangle inequality, we obtain
\[
\alpha^{1/p}\norm{u}_p
\le \norm{\Omega^\top u}_p
\le \norm{\Omega^\top(r_\star+u)}_p + \norm{\Omega^\top r_\star}_p
\le 2\beta^{1/p}\norm{r_\star}_p.
\]
Therefore
$
\norm{u}_p \le 2\left(\beta/\alpha\right)^{1/p}\norm{r_\star}_p.
$
Finally,
\[
\norm{A\widetilde x-b}_p
= \norm{r_\star+u}_p
\le \norm{r_\star}_p + \norm{u}_p
\le \left(1 + 2\left(\frac{\beta}{\alpha}\right)^{1/p}\right)\norm{r_\star}_p.
\]
The claim follows.
\end{proof}

The point of~\cref{thm:deterministic-lp} is that injectivity on $\range(A)$ controls the perturbation in $\range(A)$,  while a single upper bound on the optimal residual controls the remaining term.  The result in~\cref{thm:deterministic-lp} does not recover the relative error bound for $\ell_2$ regression when $p=2$ because the assumptions in (i) and (ii) are weaker than an OSE property on the sketch.  

In the $p$-OSI setting,  assumption (ii) is supplied in expectation by $p$-isotropy, so a Markov argument yields a probabilistic result. 

\begin{corollary}\label{cor:probabilistic-lp}
Let $A\in\R^{n\times d}$, $b\in\R^n$, and suppose that $\rank(A)\le r$. Let $\Omega$ be an $(r,\alpha,\rho)$-OSI${}_p$. Then, for every $t\ge1$,  the sketch-and-solve estimator satisfies
\[
\Prob\!\left\{ \norm{A\widetilde x-b}_p \le \left(1+2\left(\frac{t}{\alpha}\right)^{1/p}\right) \min_{x\in\R^d} \norm{Ax-b}_p \right\} \ge 1-\rho-t^{-1}.
\]
In particular, if $0<\delta<1$, $\rho\le \delta/2$, and $t=2/\delta$, then
\[
\Prob\!\left\{ \norm{A\widetilde x-b}_p \le \left(1+2\left(\frac{2}{\alpha\delta}\right)^{1/p}\right) \min_{x\in\R^d} \norm{Ax-b}_p \right\} \ge 1-\delta.
\]
\end{corollary}

\begin{proof}
Let $r_\star = Ax_\star-b$ be an optimal residual, where $x_\star\in \argmin_{x\in\R^d} \norm{Ax-b}_p$.  If $r_\star = 0$, then $b\in\range(A)$, and on the injectivity event on $\range(A)$ we get exact recovery, so the claim is trivial. Hence, assume $\norm{r_\star}_p>0$. By $p$-isotropy, $\E\|\Omega^\top r_\star\|_p^p = \|r_\star\|_p^p.$
Markov's inequality gives
\[
\Prob\!\left\{ \norm{\Omega^\top r_\star}_p^p > t\,\norm{r_\star}_p^p \right\}
\le
\frac{\E\left[\norm{\Omega^\top r_\star}_p^p\right]}{t\norm{r_\star}_p^p}
=
 t^{-1}.
\]
By the injectivity part of the $(r,\alpha,\rho)$-OSI${}_p$ hypothesis, with probability at least $1-\rho$ we also have\footnote{Since the OSI definition is for every fixed $r$-dimensional subspace,  one technically needs to extend $\range(A)$ to a $r$-dimensional subspace $V$ if ${\rm rank}(A)<r$.  The injectivity event on $V$ implies the same inequality on $\range(A)$.} 
\[
\norm{\Omega^\top v}_p^p \ge \alpha\norm{v}_p^p
\qquad \text{for all } v\in\range(A).
\]
On the intersection of these two events, \Cref{thm:deterministic-lp} applies with $\beta=t$. A union bound yields the claim.
\end{proof}

Thus,  the OSI${}_p$ sketch property guarantees that the sketch-and-solve estimator achieves a constant factor approximation for the $\ell_p$ regression problem.  

\section{Conclusion}
The OSI property introduced in~\cite{CamanoEtAl2025} is strong enough to deliver constant-factor guarantees for randomized linear algebra, but it is too weak on its own to support OSE-style relative-error bounds for sketch-and-solve least squares or randomized SVD. The missing ingredient is upper control on the optimal residual or tail component. When injectivity is strengthened on the relevant augmented subspaces, one recovers near-relative-error bounds.

\section*{Acknowledgments}
We thank the Simons Institute for the Theory of Computing for supporting the workshop on linear systems and eigenvalue problems, where we first learned about the OSI property of a sketch.  We also thank Raphael Meyer for being the scribe for Problems~5.1--5.3 in~\cite{AmselEtAl2026}. During the preparation of this manuscript, A.T. used GPT-5.4 Pro to polish the writing and check the manuscript for errors. GPT-5.4 Pro was also used to write the figure-generation code. This code and all other AI-assisted output were carefully scrutinized and verified by the authors. The authors take full responsibility for the content of the manuscript.

\bibliographystyle{siam}
\bibliography{references}

\end{document}